\documentclass[a4paper, 11pt]{article}

\usepackage[utf8]{inputenc}
\usepackage{hyperref}
\usepackage{amsthm}
\usepackage{amssymb}
\usepackage{newtxtext}
\usepackage[cochineal]{newtxmath}
\usepackage{microtype}
\usepackage[margin=3cm]{geometry}
\usepackage{amsmath}
\usepackage{graphicx}
\usepackage{semantic}
\usepackage{ebproof}
\usepackage{xcolor}
\usepackage{url}

\usepackage[english]{babel}

\usepackage[language=australian, bibstyle=mybooknumeric, citestyle=numeric]{biblatex}
\usepackage{enumerate}
\usepackage{verbatim}

\newcommand{\C}{\mathbb{C}}

\newtheorem{theorem}{Theorem}[section]

\theoremstyle{definition}
\newtheorem{definition}[theorem]{Definition}
\theoremstyle{definition}
\newtheorem{remark}[theorem]{Remark}
\theoremstyle{definition}
\newtheorem{example}[theorem]{Example}

\title{Gröbner bases and final polynomials}
\author{Peter Lundgaard \thanks{Dept. of Mathematics, Aarhus University. \href{mailto:lundgaardpeter@gmail.com}{lundgaardpeter@gmail.com}} \hspace{3cm}
Andreas Bøgh Poulsen \thanks{Dept. of Mathematics, Aarhus University. \href{mailto:bpou@outlook.dk}{bpou@outlook.dk}}}
\date{}

\begin{document}
\maketitle
\begin{abstract}
    In \cite{MR0997071} Sturmfels linked the  Hilbert Nullstellensatz to 
    Gr\"obner bases through final polynomials. In (loc.\ cit.) it was claimed
    that  final polynomials always appear in a lexicographic Gr\"obner basis
    of a certain ideal. In this paper, we give a 
    counterexample to this claim. We also show how the introduction of an
    extra variable restores the claim in a deformed setup, which we
    call extended final polynomials.
\end{abstract}

\section{Final polynomials}

Suppose that 
$F = \{f_1, \dots, f_k\} \subset
K[x_1, \dots, x_n]$, where $K$ is a field.
We recall the
definition of final polynomials from \cite[\S 1]{MR0997071}. 

\begin{definition}
    A polynomial $p \in K[x_1, \dots, x_n, y_1, \dots, y_k]$ is called \emph{final} for $F$ if
    \begin{enumerate}
        \item $p(x_1, \dots, x_n, f_1, f_2, \dots, f_k) = 0$ and
        \item $p(x_1, \dots, x_n, 0, 0, \dots, 0) \in K \setminus \{0\}$.
    \end{enumerate}
\end{definition}

If $K$ is algebraically
closed, then the (weak) Hilbert Nullstellensatz says that $V(F) := \{v\in K^n \mid f_1(v) = \cdots = f_k(v) = 0\}=\emptyset$ if and only if
the ideal, $\langle F \rangle$, generated by $F$ contains $1$.

If there exists a final polynomial for $F$, then $V(F)  = \emptyset$. On the other hand if 
$1\in \langle F \rangle$ and $\lambda_1 f_1 + \cdots + \lambda_k f_k = 1$ for $\lambda_1, \dots, \lambda_k\in K[x_1, \dots, x_n]$, then
$p = 1 - \lambda_1 y_1 - \cdots \lambda_k y_k$ is a final polynomial for 
$F$.


\section{Bernd's conjecture}

In \cite[Theorem 3.1]{MR0997071} and \cite[Theorem 6.2]{MR1122641}, Sturmfels claimed the result below without proof.

\begin{theorem}\label{FinPol : 3}
    Let $K$ be an algebraically closed field and $G$ a Gr\"obner basis of $I = \langle y_1 - f_1, y_2 - f_2, \dots, y_k - f_k \rangle \subset K[x_1, \dots, x_n, y_1, \dots, y_k]$ for the lexicographic order $x_1 > x_2 > \dots > x_n > y_1 > y_2 > \dots > y_k$. If $V(F) = \emptyset$, then $G$ contains a final polynomial for $F$.
\end{theorem}

In private communication \cite{email}, Sturmfels suggests a proof using specialization of Gr\"obner bases, but is open to counterexamples to Theorem \ref{FinPol : 3}, which he coined Bernd's conjecture (loc.\ cit.). After generating several computer examples supporting Bernd's conjecture, we finally found the counterexample below.

\begin{example}\label{FinalPolEx : 1}
    Let $F = \{f_1, f_2, f_3\} \subset \C[x_1, x_2, x_3]$ where $f_1 = x_2 + x_3$, $f_2 = x_2 x_3$ and $f_3 = x_1 x_3 + 1$. The reduced Gröbner basis of $\langle F \rangle$ is $\{1\}$, so $V(F)=\emptyset$. However, $I = \langle y_1 - f_1, y_2 - f_2, y_3 - f_3 \rangle \subset \C[x_1, x_2, x_3, y_1, y_2, y_3]$ has the reduced Gröbner basis (w.r.t.\ the lexicographic order $x_1 > x_2 > x_3 > y_1 > y_2 > y_3$)
    \begin{gather*}
    G = \{ x_1 y_2 + x_3 y_3 - x_3 - y_1 y_3 + y_1, x_3^2 - x_3 y_1 + y_2, x_2 + x_3 - y_1, x_1 x_3 - y_3 + 1\}.
    \end{gather*}
    Putting $y_1 = y_2 = y_3 = 0$, $G$ specializes to 
    \[ \{ -x_3, x_3^2, x_2 + x_3, x_1 x_3 + 1\}\]
    which does not contain a constant. Therefore $G$ does not contain a final polynomial. Even simpler
    counterexamples exist if one relaxes the condition that $G$ is reduced.
\end{example}

\section{Extended final polynomials}

In this section, we modify the statement of Theorem~\ref{FinPol : 3} to produce a Gröbner basis for a certain ideal with respect to a monomial order giving rise to a final polynomial.
In a sense we add a variable to deform the original statement.
We introduce extended final polynomials.

\begin{definition}\label{ScaledFinPolDef : 1}
    A polynomial $p\in K[z, x_1, \dots, x_n,y_1, \dots, y_k]$ is called an \emph{extended final polynomial} for $F$ if 
    \begin{enumerate}
        \item $p(z, x_1, \dots, x_n, zf_1,zf_2,\dots, zf_k) = 0$ and
        \item $p(z, x_1, \dots, x_n, 0,\dots, 0) = c z \text{ for some $c \in K \setminus \{0\}$}$.
    \end{enumerate}
\end{definition}
\begin{remark}\label{FinPolScalFinPolConnection : 1}
    If $p$ is an extended final polynomial for $F$, then $p(1, x_1, \dots, x_n, y_1, \dots, y_k)$ is a final polynomial for $F$. If $1 = \lambda_1 f_1 + \cdots + \lambda_k f_k$ for $\lambda_i\in K[x_1, \dots, x_n]$, then
    $$
    p = z - \lambda_1 y_1 - \cdots - \lambda_k y_k
    $$
    is an extended final polynomial for $F$.
    If $p$ is an extended final polynomial, then
$$
p\in \widehat{I} := \left\langle y_1 - zf_1,y_2 - zf_2, \dots, y_k - zf_k \right\rangle \subset K[z, x_1, \dots, x_n, y_1, \dots, y_k].
$$
\end{remark}


\begin{theorem}\label{ScaledFinPolTheorem}
    Let $K$ be an algebraically closed field and $G$ a Gröbner basis of $\widehat{I}$ for a monomial order satisfying that
    monomials in $K[z, x_1, \dots, x_n, y_1, \dots, y_k]\setminus K[x_1, \dots, x_n, y_1, \dots, y_k]$ are greater than monomials in $K[x_1, \dots, x_n, y_1, \dots, y_k]$. If $V(F) = \emptyset$, then $G$ contains an extended final polynomial for $F$.
\end{theorem}
\begin{proof}
    Suppose $V(F) = \emptyset$. Then by the Nullstellensatz,
    \[1 = \sum_{i=1}^k \lambda_i f_i\]
    for suitable $\lambda_1, \dots, \lambda_k \in K[x_1, \dots, x_n]$. This gives the extended final polynomial 
    $$
    p= z - \sum_{i=1}^k \lambda_i y_i\in \widehat{I}.
    $$
    Since $G$ is a Grobner basis, there 
    is some $g\in G$ with leading monomial dividing
    $z$. We may assume that the leading monomial of $g$ is $z$, as $1\not\in \widehat{I}$. Since $g\in \widehat{I}$,
    \begin{align*}
        g &= \sum_{i=1}^k \mu_i (y_i - z f_i ) \\
        &= \underbrace{z\left(\sum_{i=1}^k \mu'_i f_i \right)}_{s_1}-\underbrace{\sum_{i=1}^k \mu''_i y_i}_{s_2}
    \end{align*}
    for suitable $\mu'_i \in K[z, x_1, \dots, x_n]$ and $\mu_i, \mu''_i\in K[z, x_1, \dots, x_n, y_1, \dots, y_k]$.
    Every monomial in $s_2$ lies in $\langle y_1, \dots, y_k\rangle$ and
    there is no cancellation of terms between $s_1$ and $s_2$. 
    Therefore $\sum_{i=1}^k\mu'_i f_i =1$, $g(z, x_1, \dots, x_n, 0, \dots, 0) = z$ and $g$ is seen to be an extended final polynomial for $F$.
\end{proof}


\begin{example}\label{ScaledFinPolEx : 1}
    Let $f_1, f_2, f_3$ be given as in Example~\ref{FinalPolEx : 1}. Computing the reduced Gröbner basis of $\widehat{I} = \langle y_1 - z f_1, y_2 - z f_2, y_3 - z f_3 \rangle$ for the lexicographic order $z > x_1 > x_2 > x_3 > y_1 > y_2 > y_3$, we confirm that it contains the extended final polynomial $z + x_1^2 y_2 - x_1 x_2 y_3 + x_1 y_1 - y_3$.
\end{example}

\section*{Acknowledgments}

This paper grew out of our master's theses \cite{PETER} and \cite{ANDREAS} at Aarhus University.
We are grateful to our advisor Niels Lauritzen for guidance and for helping us prepare this article. 

\printbibliography{}
\end{document}